\documentclass[12pt]{article}
\usepackage{amsfonts,amsthm,mathrsfs,indentfirst}
\usepackage[left=3cm,right=3cm,top=2.5cm,bottom=2.5cm]{geometry}
\usepackage{graphics}
\usepackage{authblk}
\usepackage[fleqn]{amsmath}

\newtheorem{thm}{Theorem}

\newtheorem{rem}{Remark}

\def\d{\mathrm{d}}
\def\p{\mathbb{P}}

\def\R{\mathbb{R}}

\def\O{\Omega}
\def\X{\mathbf{X}}
\def\L{\mathscr{L}}

\allowdisplaybreaks
\begin{document}
\title{Constructing Cubature Formulas of Degree 5 with Few Points  \thanks{This project is supported by
NNSFC (Nos.109111440268,10801023,11171052,61033012) and ``the Fundamental
Research Funds for the Central Universities''.}}
\author[a]{Zhaoliang Meng\thanks{Corresponding author: mzhl@
dlut.edu.cn}}
\author[a,b]{Zhongxuan Luo}
\affil[a]{\it \small School of Mathematical Sciences,Dalian
University of Technology, Dalian, 116024, China}
\affil[b]{\it \small School of Software, Dalian
University of Technology, Dalian, 116620, China}

\maketitle
\begin{abstract}
This paper will devote to construct a family of fifth degree cubature formulae
for $n$-cube with symmetric measure and $n$-dimensional spherically
symmetrical region. The formula for $n$-cube contains at most $n^2+5n+3$
points and for $n$-dimensional spherically
symmetrical region contains only $n^2+3n+3$ points. Moreover, the numbers can be
reduced to $n^2+3n+1$ and $n^2+n+1$ if $n=7$ respectively, the latter of which is
minimal.
\\[6pt]
\end{abstract}

\section{Introduction}
In this article, the fifth degree formula with respect to the following
integral
\begin{eqnarray}
    \L(f)=\int_{\O}\rho(\X)f(\X)\d\X
    \label{eq:int_prob}
\end{eqnarray}
will be considered where the region $\O\subset \R^n$ and $\rho(\X)\geq 0$ for
$\X\in \O$. We study two cases: one of which
requires $\Omega=[-1,1]^n$ and $\rho(\X)=\rho_1(x_1)\ldots\rho_n(x_n)$ where
$\rho_i(x_i)$ is symmetrical;
and the other of which requires that $\L$ is a spherically symmetric integral
(see the definition of section 3 or refer to \cite{Hirao2009}).

There was a large number of numerical integration literature to
appear in the past several decades. Much of early work is largely
contained in Stroud's book\cite{1971-Stroud-p-} and Mysovskikh's
book\cite{1981-Mysovskikh-p-}. For a recent review, see
\cite{1993-Cools-p309-326,Ronald1999,Ronald2002,Ronald2003}. An on-line database containing many of the
best known cubature formulas is described in Cools' website
http://nines.cs.kuleuven.be/ecf.

Among all the cubature formulae, those with few points received most
attention. Without any special assumption, a general lower bound was
given as follows (see \cite{1997-Cools-p1-54} and the reference
therein): for a cubature formula of degree $2k$ or $2k+1$, the
number of the points $N\geq \dim\p_n^k$, where $\p_n^k$ denotes the
polynomial space of dimension $n$ with degree no more than $k$. For
even degrees, no greater lower bound was given until now, see Cools'
survey paper \cite{1997-Cools-p1-54}. However, for odd degrees
$2k+1$, the lower bound for centrally symmetrical integral can be
improved as follows (see M\"oller \cite{Moller1979}):
\begin{eqnarray}
    N\geq 2\dim \p_n^k-\begin{cases}
        0 & \text{if $k$ is odd},\\
        1 & \text{if $k$ is even},\\
    \end{cases}
    \label{eq:minimal_points}
\end{eqnarray}
or explicitly,
\begin{eqnarray}
    N\geq
    \begin{cases}\displaystyle
        \binom{n+k}{n}+\sum_{s=1}^{n-1}2^{s-n}\binom{s+k}{s} &
        \text{if $k$ is odd,}\\ \displaystyle
        \binom{n+k}{n}+\sum_{s=1}^{n-1}(1-2^{s-n})\binom{s+k-1}{s} &
        \text{if $k$ is even.}\\
    \end{cases}
    \label{eq:minimal_points_explicit}
\end{eqnarray}
Particularly, for the fifth degree case,
\begin{eqnarray*}
    N\geq n^2+n+1.
\end{eqnarray*}

However, there is a big gap between the theoretical lower bound and
the numbers of points in the known cubature formulae, especially for
the higher dimensional case.
 And the gap becomes bigger
and bigger as the dimension and the precision degree increase.
So
far, only for special dimension, the minimal cubature rules can be
obtained. Noskov and Schmid \cite{Noskov2004} derived a necessary
and sufficient condition of existence of the minimal cubature rules
with respect to the integrals over the ball.

The cubature problem on the product region received most attention and much
interesting works appeared, for recent results, see \cite{2011-Meng-p2036-2043} for the third
degree case, \cite{2001-Cools-p15-32,2006-Omelyan-p190-204} for the higher
degree case and the references therein. Besides, a good choice to construct cubature formula
is to employ Smolyak's method
\cite{2000-Barthelmann-p273-288,1998-Gerstner-p209-232,1999-Novak-p499-522,1963-Smolyak-p123-123}.
However, the number of the points is much greater than the
theoretical lower bound. Very recently, starting from Smolyak
method, Hinrichs and Novak \cite{2007-Novak-p1357-1372} presented a
kind of formula with $n^2+7n+1$ points for fifth degree case and
$(n^3+21n^2+20n+3)/3$ points for seventh degree case, which is very close to
the theoretical lower bound. Before them,
except Lu's formula \cite{2005-Lu-p613-624}, the best upper bounds were of the
form ``$\approx 2n^2$'' and ``$\approx 4n^3/3$'' respectively, where
``$\approx$'' denotes the strong
equivalence of sequences. Lu's formula  uses  $n^2+3n+3$ points just for entire $n$-dimensional space with Gaussian weight
function, which is the minimal point formula for the general
dimension among all the known formulae until now.

In this paper we shall present a method to construct fifth degree
formulae with at most $n^2+5n+3$ points for the integrals with the
product form and with at most $n^2+3n+3$ points for the spherically
symmetrical integrals, which employ fewer points than Hinrichs and
Novak's formulae \cite{2007-Novak-p1357-1372} and are more close to
the theoretical lower bound. This paper proceeds as follows. Section
2 will present a method to construct fifth formula for $n$-cube with
symmetrical product measure. And section 3 will devote to the
construction of cubature formulae for spherically symmetrical
integrals.  Section 4 will give some numerical integration
rules by our method. Finally, section 5 will conclude our results.
\section{Fifth degree formula for symmetrical product measure}
We begin with a well-known formula with respect to the surface of sphere:
$$U_n=\{\X\in\R^n:x_1^2+x_2^2+\ldots+x_n^2=1\}.$$ Due to Stroud
\cite{1971-Stroud-p-},
the monomial integrals for $\rho(x_1,x_2,\ldots,x_n)=1$ are
\begin{eqnarray*}
    &&V=\int_{U_n}\d\X=\frac{2\pi^{n/2}}{\Gamma(n/2)},\\
    &&\int_{U_n}x_1^{\alpha_1}x_2^{\alpha_2}\ldots x_n^{\alpha_n}\d
    \X=\frac{2\Gamma[(\alpha_1+1)/2]\ldots\Gamma[(\alpha_n+1)/2]}{\Gamma[(n+\alpha_1+
    \alpha_2+\ldots+\alpha_n)/2]},
\end{eqnarray*}
provided all the $\alpha_i$ are even; otherwise the integral is zero.

Mysovskikh \cite{1981-Mysovskikh-p-} derives a cubature formula of degree 5 ($n\geq 4$) as follows:
\begin{equation}
    \L_1(f)=\int_{U_n}f(\X)d\X\approx
    Q(f)=A\sum_{j=1}^{n+1}\big[f(\mathbf{a}^{(j)})+f(-\mathbf{a}^{(j)})\big]
    +B\sum_{j=1}^{n(n+1)/2}\big[f(\mathbf{b}^{(j)})+f(-\mathbf{b}^{(j)})\big]
    \label{eq:cub_Mysov}
\end{equation}
where
\begin{eqnarray}
    &&\mathbf{a}^{(r)}=(a_1^{(r)},a_2^{(r)},\ldots,a_n^{(r)}),\quad
    r=1,2,\ldots,n+1,\\
    &&a_i^{(r)}=\begin{cases}
        -\sqrt{\frac{n+1}{n(n-i+2)(n-i+1)}},& i<r,\\
        \sqrt{\frac{(n+1)(n-r+1)}{n(n-r+2)}},& i=r,\\
        0,& i>r,
    \end{cases}
\end{eqnarray}
and
\begin{eqnarray}
    \{\mathbf{b}^{(j)}=(b_1^{(j)},b_2^{(j)},\ldots,b_n^{(j)})\}\equiv
    \Big\{\sqrt{\frac{n}{2(n-1)}}(\mathbf{a}^{(k)}+\mathbf{a}^{(l)}):k<l,\ l=1,2,\ldots,n+1\Big\}.
\end{eqnarray}
The corresponding coefficients are
\begin{eqnarray*}
    A&=& n(7-n)V/[2(n+1)^2(n+2)],\\
    B&=& 2(n-1)^2V/[n(n+1)^2(n+2)].
\end{eqnarray*}

We firstly construct a cubature formula of degree five for
the integral with product form:
\begin{equation}
    \L(f)=\int_{[-1,1]^n}\rho_1(x_1)\rho_2(x_2)\ldots\rho_n(x_n)f(\X)\d\X,
    \label{eq:original_integral}
\end{equation}
where
\begin{eqnarray*}
    \int_{-1}^{1}\rho_i(x_i)\d x_i=1\quad\text{and}\quad
    \int_{-1}^{1}\rho_i(x_i)x_i^{2k+1}\d x_i=0 \quad \text{for any nonnegative
    integer $k$}.
\end{eqnarray*}

Define a diagonal matrix $T=\mbox{diag}\{a_{11},a_{22},\ldots,a_{nn}\}$, where
$a_{ii}=\sqrt{\frac{\L(x_i^2)}{\sqrt{\gamma\cdot\L_1(x_1^2x_2^2)}}}$ with a
positive parameter $\gamma$. Let
$\tilde{\mathbf{a}}^{(j)}=T\mathbf{a}^{(j)}$,
$\tilde{\mathbf{b}}^{(k)}=T\mathbf{b}^{(k)}$, $\tilde{A}=\gamma
A,\tilde{B}=\gamma B$ and
\begin{equation}
    \tilde{Q}(f)=\tilde{A}\sum_{j=1}^{n+1}\big[f(\tilde{\mathbf{a}}^{(j)})+f(-\tilde{\mathbf{a}}^{(j)})\big]
    +\tilde{B}\sum_{j=1}^{n(n+1)/2}\big[f(\tilde{\mathbf{b}}^{(j)})+f(-\tilde{\mathbf{b}}^{(j)})\big].
    \label{eq:cub_Mysov1}
\end{equation}
Then
\begin{eqnarray*}
    \tilde{Q}(x_i)&=& 0\\
    \tilde{Q}(x_ix_j)&=&\tilde{A}\sum_{r=1}^{n+1}\big[2a_{ii}a_i^{(r)}\cdot
    a_{jj}a_{j}^{(r)}\big]+\tilde{B}\sum_{r=1}^{n(n+1)/2}\big[2a_{ii}b_i^{(r)}\cdot
    a_{jj}b_{j}^{(r)}\big]\\
    &=&a_{ii}a_{jj}\Big[ \tilde{A}\sum_{r=1}^{n+1}\big[2a_i^{(r)}\cdot
    a_{j}^{(r)}\big]+\tilde{B}\sum_{r=1}^{n(n+1)/2}\big[2b_i^{(r)}\cdot
    b_{j}^{(r)}\big]\Big]\\
    &=& \sqrt{\frac{\L(x_i^2)\L(x_j^2)}{\gamma\L_1(x_1^2x_2^2)}}\cdot
    \gamma\L_1(x_ix_j)\\
    &=& \begin{cases}
        \dfrac{\sqrt{\gamma}\L_1(x_i^2)}{\sqrt{\L_1(x_1^2x_2^2)}}\cdot \L
        (x_i^2), & \text{for}\quad i=j;\\
        0, & \text{for}\quad i\neq j;\\
    \end{cases}\\
        &=& \begin{cases}
{\frac {\sqrt {2}{\pi }^{\frac{n}{4}} \left( n+2 \right) }{2\sqrt {
    \Gamma( \frac{n}{2}+2 ) }}\cdot\sqrt{\gamma}}
        \cdot \L
        (x_i^2), & \text{for}\quad i=j;\\
        0, & \text{for}\quad i\neq j;\\
    \end{cases}
\end{eqnarray*}
\begin{eqnarray*}
    \tilde{Q}(x_ix_kx_l)&=&\tilde{A}\sum_{r=1}^{n+1}\big[a_{ii}a_i^{(r)}\cdot
    a_{kk}a_{k}^{(r)}\cdot a_{ll}a_{l}^{(r)}-a_{ii}a_i^{(r)}\cdot
    a_{kk}a_{k}^{(r)}\cdot a_{ll}a_{l}^{(r)}\big]\\
    &&+\tilde{B}\sum_{r=1}^{n(n+1)/2}\big[a_{ii}b_i^{(r)}\cdot
    a_{kk}b_{k}^{(r)}\cdot a_{ll}b_{l}^{(r)}-a_{ii}b_i^{(r)}\cdot
    a_{kk}b_{k}^{(r)}\cdot a_{ll}b_{l}^{(r)}\big]\\
    &=&0
\end{eqnarray*}
\begin{eqnarray*}
    \tilde{Q}(x_ix_jx_kx_l)&=&\tilde{A}\sum_{r=1}^{n+1}\big[2a_{ii}a_i^{(r)}\cdot
    a_{jj}a_j^{(r)}\cdot a_{kk}a_k^{(r)}\cdot a_{ll}a_l^{(r)}\big]\\
    &&+\tilde{B}\sum_{r=1}^{n(n+1)/2}\big[2a_{ii}b_i^{(r)}\cdot
    a_{jj}b_j^{(r)}\cdot a_{kk}b_k^{(r)}\cdot a_{ll}b_l^{(r)}\big]\\
    &=&a_{ii}a_{jj}a_{kk}a_{ll}\Big[\tilde{A}\sum_{r=1}^{n+1}\big[2a_i^{(r)}
    a_j^{(r)}a_k^{(r)}a_l^{(r)}\big]
    +\tilde{B}\sum_{r=1}^{n(n+1)/2}\big[2b_i^{(r)}
    b_j^{(r)}b_k^{(r)}b_l^{(r)}\big]\Big]\\
    &=& \frac{\sqrt{\L(x_i^2)\L(x_j^2)\L(x_k^2)\L(x_l^2)}}{\L_1(x_1^2x_2^2)}\cdot
    \L_1(x_ix_jx_kx_l)\\
    &=& \begin{cases}
        \L (x_i^2x_k^2), & \text{for}\quad i=j\ \text{and}\ k=l;\\
        \left(\L (x_i^2)\right)^2\cdot\frac{\L_1(x_i^4)}{\L_1(x_1^2x_2^2)}, &
        \text{for}\quad i=j=k=l;\\
        0, & \text{others}
    \end{cases}\\
    &=& \begin{cases}
        \L (x_i^2x_k^2), & \text{for}\quad i=j\ \text{and}\ k=l;\\
        3\left(\L (x_i^2)\right)^2, &
        \text{for}\quad i=j=k=l;\\
        0, & \text{others}
    \end{cases}
\end{eqnarray*}
and similarly $\tilde{Q}[x_ix_jx_kx_lx_m]=0$. Define $\tilde{\L}=\L-\tilde{Q}$
and compute the moments as follows:
\begin{eqnarray*}
    &&\tilde{\L}(1)=1-\gamma\cdot V,\\
    &&\tilde{\L}(x_i)=0,\\
    &&\tilde{\L}(x_ix_j)=\begin{cases}
        0, & \text{for}\quad i\neq j,\\
        \L(x_j^2)\cdot\left(
    1-{\frac {\sqrt {2}{\pi }^{\frac{n}{4}} \left( n+2 \right) }{2\sqrt {
    \Gamma( \frac{n}{2}+2 ) }}\cdot\sqrt{\gamma}}
    \right), & \text{for}\quad i=j,\\
    \end{cases}\\
    &&\tilde{\L}(x_ix_jx_k)=0,\\
    &&\tilde{\L}(x_ix_jx_kx_l)=\begin{cases}
        \L(x_j^4)-3\left(\L(x_j^2)\right)^2,&
        \text{for}\quad i=j=k=l,\\
        0, & \text{others,}\\
    \end{cases}\\
    &&\tilde{\L}(x_ix_jx_kx_lx_m)=0.
    \label{eq:moment}
\end{eqnarray*}
\begin{rem}
    If $\tilde{R}(f)$ is a cubature formula of degree five with respect to
    $\tilde{\L}$, then $\L\approx \tilde{Q}+\tilde{R}$ is a
    cubature formula of degree five with respect to $\L$.
\end{rem}

Now it is the time to state our main result.
\begin{thm}\label{th:mainresult}
There exists a cubature formula $C_n$ for $\L$ of degree five which
uses at most $n^2+5n+3$ points.
\end{thm}
\begin{proof}
    Suppose that an integration formula $\tilde{R}$ of degree five for $\tilde{\L}$ is given by
    \begin{equation}
        \tilde{R}(f)=\sum_{i=1}^n\sum_{j=1}^2w_{ij}
        f(v_{i,j}\mathbf{e}_i)+Cf(0,0,\ldots,0)
     \label{eq:cub_form}
\end{equation}
where $\mathbf{e}_i=(0,\ldots,0,1,0,\ldots,0)$ is a vector of dimension
$n$ and the $i$th element one is the only nonzero element. To enforce
polynomial exactness of degree five, it suffices to require
\eqref{eq:cub_form} to satisfy the following equations:
\begin{equation}\label{eq:momentproblem}
    \begin{split}
    &\sum_{j=1}^2
    w_{ij}(v_{i,j})^m=\tilde{\L}(x_i^m),\quad
    m=1,2,3,4,5,i=1,2,\ldots,n,\\
    &\sum_{i=1}^n\sum_{j=1}^2w_{ij}+C=\tilde{\L}(1).
    \end{split}
\end{equation}

By a proper parameter $\gamma$, we can make
$\tilde{\L}(x_i^2)\cdot\tilde{\L}(x_i^4)>0$.
Thus we get the required cubature formula $\tilde{R}$ with
\begin{eqnarray*}
    &&v_{i,1}=\sqrt{\frac{\tilde{\L}(x_i^4)}{\tilde{\L}(x_i^2)}},\quad
    v_{i,2}=-v_{i,1},\quad
    w_{i,1}=w_{i,2}=\frac{\big(\tilde{\L}(x_i^2)\big)^2}{2\tilde{\L}(x_i^4)},\
    i=1,2,\ldots,n,\\
    &&C=\tilde{L}(1)-\sum_{i=1}^n(w_{i,1}+w_{i,2}).
\end{eqnarray*}
The corresponding cubature formula can be written as
\begin{equation}
    \begin{split}
    \L(f)\approx &\tilde{A}\sum_{j=1}^{n+1}\big[f(\tilde{\mathbf{a}}^{(j)})+f(-\tilde{\mathbf{a}}^{(j)})\big]
    +\tilde{B}\sum_{j=1}^{n(n+1)/2}\big[f(\tilde{\mathbf{b}}^{(j)})+f(-\tilde{\mathbf{b}}^{(j)})\big]\\
    &+\sum_{i=1}^n\sum_{j=1}^2w_{ij}
        f(v_{i,j}\mathbf{e}_i)+Cf(0,0,\ldots,0).
    \end{split}
    \label{eq:final_cubature}
\end{equation}
The total number of the points is at most $(n^2+3n+2)+(2n+1)=n^2+5n+3$.
\end{proof}

\begin{rem}
    To ensure all the points inside the integration region, it is required
    that
    \begin{eqnarray}
        a_{ii}\leq 1,\quad\text{and}\quad
        0<\frac{\tilde{\L}(x_i^4)}{\tilde{\L}(x_i^2)}\leq 1.
        \label{eq:constrained_condition}
    \end{eqnarray}
    If $\L(x_j^4)<3\L(x_j^2)^2$, then
    \begin{equation}\label{eq:condition}
        \begin{split}
        & \gamma\geq 2\pi^{-\frac{n}{2}}{\Gamma\left(n/2+2 \right)
         \big(\L(x_i^2)\big)^2},\quad
        \sqrt{\gamma}>{\frac {\sqrt {2\Gamma  \left( n/2+2 \right) }}{(n+2){\pi
        }^{n/4}}},\\
        & \sqrt{\gamma}\geq
        \frac{\left(\L(x_i^2)-\L(x_i^4)+3\L(x_i^2)^2\right)}{\L(x_i^2)}\cdot\frac {
        \sqrt {2\Gamma\left( n/2+2 \right) }}{\left(n+2\right){\pi}^{n/4}}
        \end{split}
    \end{equation}
    which always has a solution. Otherwise
    \begin{eqnarray*}
        \frac{\big(\L(x_i^2)\big)^2}{\L_1(x_1^2x_2^2)}\leq \gamma\leq
        \frac{2
    \Gamma( \frac{n}{2}+2 ) } {{\pi }^{\frac{n}{2}} \left( n+2 \right)^2 }
    \end{eqnarray*}
    usually does not derive solution which means some points lie outside
    of the region.

    Particularly, if $\rho_i(x_i)=(1-x_i^2)^{\alpha_i}(\alpha_i>-1)$, then
    $$
    \L(x_j^4)-3\L(x_j^2)^2=-\frac{6}{(\alpha_i/2+3)^2(\alpha_i/2+5)}<0
    $$
    implies that the solution for $\gamma$ always exsits.
\end{rem}
\begin{rem}
    Since $A$ vanished when $n=7$, the total number of the points is
    $n^2+3n+1$ if $n=7$.
\end{rem}
\begin{rem}
    If
    $$1-{\frac {\sqrt {2}{\pi }^{\frac{n}{4}} \left( n+2 \right) }{2\sqrt {
    \Gamma( \frac{n}{2}+2 ) }}\cdot\sqrt{\gamma}}=0,\quad\text{or}\quad
    \gamma=\frac{\Gamma(\frac{n}{2}+1)}{\pi^{n/2}(n+2)},
    $$
    then $\tilde{\L}(x_j^2)=0$ and thus
    \begin{eqnarray*}
        \gamma\cdot Q(f)+(1-\gamma\cdot V)f(0,0,\ldots,0)
    \end{eqnarray*}
    is a cubature formula of degree three. Furthermore, it is easy to
    check that
    \begin{eqnarray*}
        \gamma\cdot Q(f)+\sum_{i=1}^n\gamma_i\frac{\partial^4f}{\partial
        x_i^4}f(0,0,\ldots,0)+(1-\sum_{i=1}^n\gamma_i-\gamma\cdot
        V)f(0,0,\ldots,0)
    \end{eqnarray*}
    is a cubature formula of degree five provided $f$ is smooth enough,
    where $$\gamma_i=\frac{\L(x_i^4)-3(\L(x_i^2))^2}{4!}.$$
\end{rem}
\section{Fifth degree formula for spherically symmetrical measure}

$\L$ is a spherically symmetrical integral functional means the region is of
the form $\{\X\in\R^n|p\leq ||\X||<q\}$ and the weight function $\rho(\X)$ is a
function of $||\X||$, where $<\X,\mathbf{Y}>=x_1y_1+\ldots+x_ny_n$ and
$||\X||=\sqrt{<\X,\X>}$ (see \cite{Hirao2009}). The classical examples are the
integrals over the ball and over the entire $\R^n$ space with Gaussian weight.
In this case, $\L(x_i^2)=\L(x_j^2)$ and $\L(x_i^2x_j^2)=\L(x_k^2x_l^2)$ for
any $i,j,k,l$. Taking
 $a_{ii}=\sqrt[4]{\frac{\L(x_1^2x_2^2)}{\gamma\cdot\L_1(x_1^2x_2^2)}}$ and the
 same mathod explained, we have
 \begin{eqnarray*}
     \tilde{Q}(x_i)&=& 0\\
     \tilde{Q}(x_ix_j)&=& \begin{cases}
{\frac {\sqrt {2}{\pi }^{\frac{n}{4}} \left( n+2 \right) }{2\sqrt {
    \Gamma( \frac{n}{2}+2 ) }}\cdot\sqrt{\gamma\cdot\L
        (x_1^2x_2^2)}}, & \text{for}\quad i=j;\\
        0, & \text{for}\quad i\neq j;\\
    \end{cases}\\
    \tilde{Q}(x_ix_jx_k)&=& 0;\\
    \tilde{Q}(x_ix_jx_kx_l)&=& \begin{cases}
        \L (x_i^2x_k^2), & \text{for}\quad i=j\ \text{and}\ k=l;\\
        3\left(\L (x_1^2x_2^2)\right), &
        \text{for}\quad i=j=k=l;\\
        0, & \text{others}.
    \end{cases}\\
    \tilde{Q}(x_ix_jx_kx_lx_m)&=&0
 \end{eqnarray*}
 and
\begin{eqnarray*}
    &&\tilde{\L}(1)=1-\gamma\cdot V,\\
    &&\tilde{\L}(x_i)=0,\\
    &&\tilde{\L}(x_ix_j)=\begin{cases}
        0, & \text{for}\quad i\neq j,\\
        \L(x_1^2)-\frac {\sqrt {2}{\pi }^{\frac{n}{4}} \left( n+2 \right) }{2\sqrt {
    \Gamma( \frac{n}{2}+2 ) }}\cdot\sqrt{\gamma\cdot\L
        (x_1^2x_2^2)}, & \text{for}\quad i=j,\\
    \end{cases}\\
    &&\tilde{\L}(x_ix_jx_k)=0,\\
    &&\tilde{\L}(x_ix_jx_kx_l)=\begin{cases}
        \L(x_1^4)-3\left(\L(x_1^2x_2^2)\right),&
        \text{for}\quad i=j=k=l,\\
        0, & \text{others,}\\
    \end{cases}\\
    &&\tilde{\L}(x_ix_jx_kx_lx_m)=0
\end{eqnarray*}
where $\tilde{\L}=\L-\tilde{Q}$.

It is a slightly different from the product case, because we have the
following
\begin{thm}
    Suppose that
    $$\L(f)=\int_{\Omega}f(\X)\rho(\X)\d\X$$ is a spherically symmetric integral functional, then
    $$\L(x_1^4)-3\L(x_1^2x_2^2)=0.$$
\end{thm}
\begin{proof}
    To calculate $\L(x_1^4)$ and $\L(x_1^2x_2^2)$, we introduce the
    transformation which was used in \cite[Page 33]{1971-Stroud-p-}
    \begin{eqnarray}
        \begin{split}
            &x_1=r\cos\theta_{n-1}\cos\theta_{n-2}\ldots\cos\theta_{2}\cos\theta_{1}\\
            &x_2=r\cos\theta_{n-1}\cos\theta_{n-2}\ldots\cos\theta_{2}\sin\theta_{1}\\
            &x_3=r\cos\theta_{n-1}\cos\theta_{n-2}\ldots\sin\theta_{2}\\
            &\ldots\ldots\ldots\ldots\\
            &x_{n-1}=r\cos\theta_{n-1}\sin\theta_{n-2}\\
            &x_n=r\sin\theta_{n-1}
        \end{split}
        \label{eq:transform}
    \end{eqnarray}
    The Jacobian of transformation \eqref{eq:transform} is
    $$
    J=r^{n-1}(\cos\theta_{n-1})^{n-2}(\cos\theta_{n-2})^{n-3}\ldots(\cos\theta_{3})^{2}(\cos\theta_{2})
    $$
    and therefore
    \begin{eqnarray*}
        \L(x_1^4)=\left\{\int_{p}^{q}+\int_{-q}^{-p}\right\}\rho(r)|r|^{n-1}r^4\d
        r\int_{-\pi/2}^{\pi/2}(\cos\theta_1)^4\d\theta_1\times I_0\\
        \L(x_1^2x_2^2)=\left\{\int_{p}^{q}+\int_{-q}^{-p}\right\}\rho(r)|r|^{n-1}r^4\d
        r\int_{-\pi/2}^{\pi/2}(\cos\theta_1)^2(\sin\theta_1)^2\d\theta_1\times I_0\\
    \end{eqnarray*}
    where
    $$
    I_0=\int_{-\pi/2}^{\pi/2}(\cos\theta_2)^5\d\theta_2\int_{-\pi/2}^{\pi/2}(\cos\theta_3)^6\d\theta_3\ldots
    \int_{-\pi/2}^{\pi/2}(\cos\theta_{n-1})^{n+2}\d\theta_{n-1}.
    $$
    Hence
    \begin{eqnarray*}
        \frac{\L(x_1^4)}{\L(x_1^2x_2^2)}=\frac{\int_{-\pi/2}^{\pi/2}(\cos\theta_1)^4\d\theta_1}{\int_{-\pi/2}^{\pi/2}(\cos\theta_1)^2(\sin\theta_1)^2\d\theta_1}=\frac{3\pi/8}{\pi/8}=3
    \end{eqnarray*}
    and therefore $\L(x_1^4)-3\L(x_1^2x_2^2)=0$. This completes the proof.
\end{proof}
Let
\begin{eqnarray*}
    \L(x_1^2)-\frac {\sqrt {2}{\pi }^{\frac{n}{4}} \left( n+2 \right) }{2\sqrt {
    \Gamma( \frac{n}{2}+2 ) }}\cdot\sqrt{\gamma\cdot\L
        (x_1^2x_2^2)}=0
\end{eqnarray*}
that is
\begin{eqnarray*}
    \gamma=\frac{(\L(x_1^2))^2}{\L(x_1^2x_2^2)}\cdot\frac{2\Gamma(n/2+2)}{\pi^{n/2}(n+2)^2}.
\end{eqnarray*}
In this case, we get a fifth degree formula with at most $n^2+3n+3$ points
\begin{eqnarray}
    \begin{split}
    C_n(f)&=\tilde{A}\sum_{j=1}^{n+1}\big[f(\tilde{\mathbf{a}}^{(j)})+f(-\tilde{\mathbf{a}}^{(j)})\big]
    +\tilde{B}\sum_{j=1}^{n(n+1)/2}\big[f(\tilde{\mathbf{b}}^{(j)})+f(-\tilde{\mathbf{b}}^{(j)})\big]\\
    &\quad +(1-\gamma\cdot
    V)f(0,0,\ldots,0).
       \end{split}
    \label{eq:Mini_cubature}
\end{eqnarray}
If $n=7$, $\tilde{A}=0$ will yields a minimal cubature formula which only
contains $n^2+n+1$ points.

\section{Numerical Results}
Assume $n\geq 4$.

\begin{itemize}
    \item Symmetric product region.
\end{itemize}

We consider the case of weight functions $\rho_1=\rho_2=\ldots=\rho_n=1/2$.

In this case,
\begin{eqnarray*}
    \tilde{\L}(x_i^4)=\L(x_j^4)-3(\L(x_j^2))^2=1/5-1/3=-2/15<0.
\end{eqnarray*}
To confirm all the points inside the region, it follows from
\eqref{eq:condition} that
\begin{eqnarray*}
    \gamma\geq \frac{2\Gamma(n/2+2)}{9\pi^{n/2}}.
\end{eqnarray*}
If we take $\gamma=\frac{2\Gamma(n/2+2)}{9\pi^{n/2}}$, then
\begin{eqnarray*}
    &&a_{ii}=1,\quad v_{i,j}=\pm\frac{\sqrt{30(n-1)}}{5(n-1)},\\
    &&w_{i,j}=-\frac{5}{108}n^2+\frac{5}{54}n-\frac{5}{108},\quad
    C=\frac{5}{54}n^3-\frac{8}{27}n^2-\frac{7}{54}n+1.
\end{eqnarray*}
Due to \eqref{eq:final_cubature}, the corresponding cubature formula can be
written as
\begin{eqnarray*}
    &&\frac{1}{2^n}\int_{[-1,1]^n}f(\X)\d\X\approx
    (\frac{5}{54}n^3-\frac{8}{27}n^2-\frac{7}{54}n+1)f(0,0,\ldots,0)\\
    &&+(-\frac{5}{108}n^2+\frac{5}{54}n-\frac{5}{108})\sum_{i=1}^n\left(
     f\left(\frac{\sqrt{30(n-1)}}{5(n-1)}\mathbf{e}_i\right)+f\left(-\frac{\sqrt{30(n-1)}}{5(n-1)}\mathbf{e}_i\right)\right) \\
     &&+\frac{(7-n)n^2}{18(n+1)^2}
    \sum_{j=1}^{n+1}\big[f(\sqrt{n/2+1}\cdot\mathbf{a}^{(j)})+f(-\sqrt{n/2+1}\cdot\mathbf{a}^{(j)})\big]\\
    &&+\frac{2(n-1)^2}{9(n+1)^2}\sum_{j=1}^{n(n+1)/2}\big[f(\sqrt{n/2+1}\cdot\mathbf{b}^{(j)})+f(-\sqrt{n/2+1}\cdot\mathbf{b}^{(j)})\big]
\end{eqnarray*}

\begin{itemize}
    \item Spherically symmetric integral
\end{itemize}

(1). Entire $n$-dimensional space with Gaussian weight function.

Simple calculation yields
\begin{eqnarray*}
    \gamma=\frac{\Gamma(n/2+1)}{\pi^{n/2}(n+2)},\quad\text{and}\quad
    a_{ii}=\sqrt{n/2+1}.
\end{eqnarray*}
Thus the fifth degree formula can be written as
\begin{eqnarray}
    \begin{split}
    &\int_{\mathbf{R}^n}e^{-x_1^2-\ldots-x_n^2}f(x_1,x_2,\ldots,x_n)\d
    x_1\ldots\d x_n\approx\frac{2\pi^{n/2}}{n+2}f(0,0,\ldots,0)\\
    &+\frac{n^2(7-n)\pi^{n/2}}{2(n+1)^2(n+2)^2}
    \sum_{j=1}^{n+1}\big[f(\sqrt{n/2+1}\cdot\mathbf{a}^{(j)})+f(-\sqrt{n/2+1}\cdot\mathbf{a}^{(j)})\big]\\
    &+\frac{2(n-1)^2\pi^{n/2}}{(n+1)^2(n+2)^2}\sum_{j=1}^{n(n+1)/2}\big[f(\sqrt{n/2+1}\cdot\mathbf{b}^{(j)})+f(-\sqrt{n/2+1}\cdot\mathbf{b}^{(j)})\big]
    \end{split}\label{lu}
\end{eqnarray}
which is same with Lu's formula \cite{2005-Lu-p613-624}.

(2). The $n$-dimensional unit ball.

Simple calculation yields
\begin{eqnarray*}
\gamma={\frac {2\left( n+4 \right) \Gamma  \left( n/2+2 \right) }{
 \left( n+2 \right) ^{3}{\pi }^{n/2}}},\quad\text{and}\quad
 a_{ii}=\sqrt{\frac{n+2}{n+4}}.
\end{eqnarray*}
Thus the fifth degree formula can be written as
\begin{eqnarray*}
    &&\int_{\mathbf{B}^n}f(x_1,x_2,\ldots,x_n)\d
    x_1\ldots\d x_n\approx\frac{8\pi^{n/2}}{n(n+2)^2\Gamma(n/2)}f(0,0,\ldots,0)\\
    &&+
    {\frac { \left(7-n \right) {n} \left( n+4 \right)\pi^{n/2} }{
 \left( n+1 \right) ^{2} \left( n+2 \right) ^{3}\Gamma(n/2)}}
    \sum_{j=1}^{n+1}\big[f(\sqrt{\frac{n+2}{n+4}}\cdot\mathbf{a}^{(j)})+f(-\sqrt{\frac{n+2}{n+4}}\cdot\mathbf{a}^{(j)})\big]\\
    &&+
{\frac {4 \left( n-1 \right) ^{2} \left( n+4 \right)\pi^{n/2} }{ n\left( n+1
 \right) ^{2} \left( n+2 \right) ^{3}\Gamma(n/2)}}
    \sum_{j=1}^{n(n+1)/2}\big[f(\sqrt{\frac{n+2}{n+4}}\cdot\mathbf{b}^{(j)})+f(-\sqrt{\frac{n+2}{n+4}}\cdot\mathbf{b}^{(j)})\big].
\end{eqnarray*}

(3). The $n$-dimensional Spherical Shell.

The region $S_n^{\text{shell}}$ is
\begin{eqnarray*}
   S_n^{\text{shell}}=\{(x_1,x_2,\ldots,x_n)| r\leq x_1^2+x_2^2+\ldots+x_n^2\leq
   1\}
\end{eqnarray*}
and weight function $\rho(\X)=1$. The monomial integrals can be found in
\cite{1971-Stroud-p-}.
 In this case,
 \begin{eqnarray*}
\gamma={\frac {2\left( n+4 \right) \Gamma  \left( n/2+2 \right) }{
 \left( n+2 \right) ^{3}{\pi }^{n/2}}},\quad\text{and}\quad
 a_{ii}=\sqrt[4]{1-r^{n+4}}\cdot\sqrt{\frac{n+2}{n+4}}.
\end{eqnarray*}
Thus the fifth degree formula can be written as
\begin{eqnarray*}
    &&\int_{S_n^{\text{shell}}}f(x_1,x_2,\ldots,x_n)\d
    x_1\ldots\d x_n\approx\frac{8\pi^{n/2}(1-r^n)}{n(n+2)^2\Gamma(n/2)}f(0,0,\ldots,0)\\
    &&+
    {\frac { \left(7-n \right) {n} \left( n+4 \right)\pi^{n/2}(1-r^n) }{
 \left( n+1 \right) ^{2} \left( n+2 \right) ^{3}\Gamma(n/2)}}
    \sum_{j=1}^{n+1}\big[f(\sqrt[4]{1-r^{n+4}}\cdot\sqrt{\frac{n+2}{n+4}}\cdot\mathbf{a}^{(j)})\\
    &&\hspace{6cm}+f(-\sqrt[4]{1-r^{n+4}}\cdot\sqrt{\frac{n+2}{n+4}}\cdot\mathbf{a}^{(j)})\big]\\
    &&+
{\frac {4 \left( n-1 \right) ^{2} \left( n+4 \right)\pi^{n/2}(1-r^n) }{ n\left( n+1
 \right) ^{2} \left( n+2 \right) ^{3}\Gamma(n/2)}}
    \sum_{j=1}^{n(n+1)/2}\big[f(\sqrt[4]{1-r^{n+4}}\cdot\sqrt{\frac{n+2}{n+4}}\cdot\mathbf{b}^{(j)})\\
    &&\hspace{6cm}+f(-\sqrt[4]{1-r^{n+4}}\cdot\sqrt{\frac{n+2}{n+4}}\cdot\mathbf{b}^{(j)})\big].
\end{eqnarray*}
The origin is outside of the region.

(4). Entire $n$-dimensional space with weight function
$\text{exp}(-\sqrt{x_1^2+x_2^2+\ldots+x_n^2})$.

Simple calculation yields
\begin{eqnarray*}
    \gamma=\frac{(n+1)\Gamma(n/2+1)}{\pi^{n/2}(n+2)(n+3)},\quad\text{and}\quad
    a_{ii}=\sqrt{(n+2)(n+3)}.
\end{eqnarray*}
Thus the fifth degree formula can be written as
\begin{eqnarray*}
    &&\int_{\mathbf{R}^n}e^{-\sqrt{x_1^2+\ldots+x_n^2}}f(x_1,x_2,\ldots,x_n)\d
    x_1\ldots\d
    x_n\approx\frac{(2n+3)\pi^{\frac{n-1}{2}}2^{n+2}\Gamma(\frac{n+1}{2})}{(n+2)(n+3)}f(0,0,\ldots,0)\\
    &&+\frac{n^2(7-n)\pi^{\frac{n-1}{2}}2^n\Gamma(\frac{n+1}{2})}{(n+3)(n+1)(n+2)^2}
    \sum_{j=1}^{n+1}\big[f(\sqrt{(n+2)(n+3)}\cdot\mathbf{a}^{(j)})+f(-\sqrt{(n+2)(n+3)}\cdot\mathbf{a}^{(j)})\big]\\
    &&+\frac{4(n-1)^2\pi^{\frac{n-1}{2}}2^n\Gamma(\frac{n+1}{2})}{(n+3)(n+1)(n+2)^2}\sum_{j=1}^{n(n+1)/2}\big[f(\sqrt{(n+2)(n+3)}\cdot\mathbf{b}^{(j)})+f(-\sqrt{(n+2)(n+3)}\cdot\mathbf{b}^{(j)})\big].
\end{eqnarray*}

\section{conclusion}
In this paper we develop a method to construct fifth degree cubature formulae
for $n$-cube with symmetric measure and $n$-dimensional spherically
symmetrical region. By a well-known cubature formula due to Mysovskikh
\cite{1981-Mysovskikh-p-}, the  cubature problems to be dealt with are
turned into a family of one-dimensional moment problems, which simplifies the construction process. Numerical results show that the number of the cubature rule constructed
by our method is minimal among all the existing cubature formulae.
Furthermore, all the cubature formulae except Lu's formula \eqref{lu} seem to
new.


\end{document}